\documentclass[10pt]{article}
\usepackage{aliascnt}
\usepackage{hyperref}
\usepackage[pdftex]{graphicx}
\usepackage{amsmath}
\usepackage{amsfonts}
\usepackage{amssymb}
\usepackage{amsthm}
\usepackage{array}
\usepackage{xy}
\usepackage{fancyhdr}
\usepackage{framed}
\usepackage{verbatim}

\theoremstyle{plain}
\newtheorem{theorem}{Theorem}
\newtheorem{lemma}[theorem]{Lemma}
\newtheorem{corollary}[theorem]{Corollary}
\newtheorem{conjecture}{Conjecture}

\theoremstyle{definition}

\newcommand{\seqnum}[1]{\href{http://oeis.org/#1}{\underline{#1}}}

\title{Chocolate Numbers}
\author{Caleb Ji, Tanya Khovanova, Robin Park, Angela Song}
\date{\today}

\begin{document}

\maketitle

\begin{abstract}
In this paper, we consider a game played on a rectangular $m \times n$ gridded chocolate bar. Each move, a player breaks the bar along a grid line.  Each move after that consists of taking any piece of chocolate and breaking it again along existing grid lines, until just $mn$ individual squares remain.

This paper enumerates the number of ways to break an $m \times n$ bar, which we call \textit{chocolate numbers}, and introduces four new sequences related to these numbers. Using various techniques, we prove interesting divisibility results regarding these sequences.
\end{abstract}

\section{Introduction}

The following puzzle involving the breaking of a chocolate bar is very famous. It appears on thousands of webpages, but we could not trace its history.

\begin{quote}
Given a rectangular $m \times n$ chocolate bar, you want to break it into the $1 \times 1$ squares. You are only allowed to break it along the grid lines and you cannot break two or more pieces at once. What is the smallest number of breaks that you need?
\end{quote}

Surprisingly, the number of breaks does not depend on the way you break: it is predetermined.  

Let us convert the puzzle to a game played by two people. We call this game the \textit{Chocolate Bar} game.  They break the chocolate in turns and the person who cannot move first loses. This game may at first seem uninteresting, as both the number of moves and the winner is predetermined. However, this game becomes more complex if we start to ask different questions. For example, how many ways are there to play this game?

This question of finding the number of ways to break a chocolate bar was suggested to us by Prof.~James Propp. In his paper, ``Games of No Strategy and Low-Grade Combinatorics'' \cite{Propp}, he counts the number of ways to play games where the winner is predetermined, or, as he calls it, \textit{games of no strategy}. As he also shows, the number of end positions provides numerous interesting combinatorial questions as well \cite{Propp}. These games now comprise a part of the vast field of \textit{impartial combinatorial games}. \cite{BCG, ANW, Siegel}.

The Chocolate Bar game can be considered as an Impartial Cutcake game. In the classic Cutcake game, played by two players, the first player is allowed to make only vertical breaks on a rectangular, gridded cake, while the second player can only make horizontal breaks. The last player able to break the cake is the winner \cite{ANW, BCG}.

In this paper, we study Chocolate Numbers, or the number of ways to play the Chocolate Bar game on multiple sizes of bars. We first introduce four new sequences that stem from Chocolate Numbers, then, through various techniques, we prove multiple divisibility and periodicity rules for these sequences.

In Section~\ref{sec:rcn} we find a recursive formula for the Chocolate Numbers, and in Section~\ref{sec:seq} we list and describe the sequences we study. Section~\ref{sec:div} proves properties of these numbers involving divisibility by 2, while Section~\ref{sec:div2} focuses on Chocolate Numbers for $2 \times n$ chocolate bars, proving further divisibility properties.  Finally, the rest of Section~\ref{sec:div2} studies this sequence's generating function and utilizes hypergeometric functions and all previous knowledge to prove or conjecture its periodicity when divided by any number.

\section{The Recursion for Chocolate Numbers}\label{sec:rcn}

Though it may initially seem surprising, it is a well-known fact that the number of breaks needed to dissect a rectangular chocolate grid is invariant however the breaks are made.

\begin{lemma}
\label{L1}
In an $m\times n$ grid, the number of breaks necessary to dissect it into $mn$ unit squares is $mn-1$.
\end{lemma}

\begin{proof}
Each break increases the number of pieces by one.  Since we begin with $1$ piece and end with $mn$ pieces, $mn-1$ breaks are required.
\end{proof}

For the rest of the paper we will study not the number of breaks, but instead the number of ways to break.

Consider an $m \times n$ rectangular grid and the number $A_{m, n}$ of distinct ways this grid can be broken along horizontal or vertical grid lines into $mn$ unit squares.  Each break divides one piece into two, and in each successive break, a single piece is chosen and is broken along a grid line it contains.  Two breaks are considered distinguishable if they break different pieces or along different grid lines.

We call the numbers $A_{m, n}$ the \textit{Chocolate Numbers}.

For example, consider a $2 \times 2$ grid. There are four ways to dissect this grid into unit squares: The first break can either be horizontal or vertical, and the second and third breaks depend on which piece of the grid one chooses next. Thus, $A_{2, 2} = 4$. 
Figure~\ref{fig:A22} shows all possible ways to break a 2 by 2 chocolate bar.

\begin{figure}[htbp]
\centering
    \includegraphics[scale=0.4]{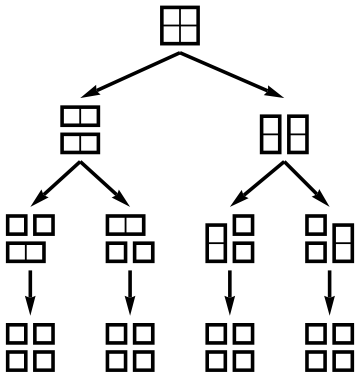}
    \caption{$A_{2, 2}=4$}\label{fig:A22}
\end{figure}

\newpage 

Note that breaking chocolate is a recursive procedure.  Indeed, since every break breaks a bar into two smaller pieces, we obtain the following recursive formula for $A_{m, n}$.

\begin{theorem}
\label{Th2}
\[A_{m, n} = \sum_{i=1}^{m-1} \binom{mn-2}{in-1} A_{i, n}A_{m-i, n} + \sum_{i=1}^{n-1} \binom{mn-2}{im-1} A_{i, m}A_{n-i, m}.\]
\end{theorem}

\begin{proof}
\textbf{Case 1:} The first move breaks the bar along a horizontal grid line.

\noindent If this break is along the $i^\text{th}$ grid line with $1\le i\le m-1$, then an $i\times n$ and an $(m-i)\times n$ bar remain.  By Lemma~\ref{L1}, the first sub-grid will take $in-1$ moves to finish, while the second will take $(m-i)n-1$ moves to finish.  Thus there are a total of $mn-2 \choose im-1$ ways to choose which moves will be on parts of which of the two grids in the remaining $mn-2$ moves.  For each of these ways, there are $A_{i, n}A_{n-i, n}$ ways to decide how the two grids are dissected.  Summing over all $i$, this gives a total of  

\[\sum_{i=1}^{m-1} {mn-2 \choose in-1}A_{i, n}A_{n-i, n}\]
ways to dissect in this case.

\textbf{Case 2:} The first move breaks the bar along a vertical grid line.

\noindent Similarly, there are

\[\sum\limits_{i=1}^{n-1} {mn-2 \choose im-1}A_{i, m}A_{m-i, m}\]
ways to dissect in this case.

Thus, the recursion for the number of ways to break an $m \times n$ bar is \[A_{m,n}=\sum\limits_{i=1}^{m-1} {mn-2 \choose in-1}A_{i, n}A_{m-i, n}+\sum\limits_{i=1}^{n-1} {mn-2 \choose im-1}A_{m, i}A_{m, n-i}.\]
\end{proof}

\begin{corollary}
\label{C3}
$A_{1, n} = A_{n, 1} = (n-1)!$.
\end{corollary}

This can also be seen by noting that on the $k^\text{th}$ move, there are $n-k$ grid lines to choose from.  Multiplying them all as $k$ ranges from $1$ through $n-1$ gives the result.

 The next corollary is clear without invoking the recursion as it is due to the symmetry of the problem.

\begin{corollary}
\label{C4}
$A_{m, n} = A_{n, m}$.
\end{corollary}

\subsection{Some values}

The values for  $A_{m,n}$ with small indices are presented in Table~\ref{tbl:Amn}.

{\small 
\begin{table}[htbp]
\begin{center}
\begin{tabular}{c | c | c | c | c | c |}
 $A_{m, n}$ & \bfseries 1 & \bfseries 2 & \bfseries 3 & \bfseries 4 & \bfseries 5\\
\hline
\bfseries 1 & 1 & 1 & 2 & 6 & 24 \\ 
\hline
\bfseries 2 & 1 & 4 & 56 & 1712 & 92800 \\ 
\hline
\bfseries 3 & 2 & 56 & 9408 & 4948992 & 6085088256 \\ 
\hline
\bfseries 4 & 6 & 1712 & 4948992 & 63352393728 & 2472100837326848 \\ 
\hline
\bfseries 5 & 24 & 92800 & 6085088256 & 2472100837326848 & 3947339798331748515840 \\ 
\hline
\end{tabular}
\caption{Values of $A_{m, n}$}\label{table:f}\label{tbl:Amn}
\end{center}
\end{table}
}

In the following section we introduce four sequences related to these numbers.

\section{Chocolate Sequences}
\label{sec:seq}

\subsection{Chocolate numbers}

The following new sequence, \seqnum{A261746} in the OEIS, lists the numbers equal to $A_{m, n}$ for some $m$, $n$ in increasing order: 1, 2, 4, 6, 24, 56, 120, 720, 1712, 5040, 9408, 40320, 92800, 362880, 
3628800, 4948992, 7918592, 39916800, 479001600, 984237056, 6085088256, $\ldots$.

The sequence of factorials, \seqnum{A000142}, is a subsequence of this sequence due to Corollary~\ref{C3}.

\subsection{Chocolate triangle}

The new sequence, \seqnum{A261964} in the OEIS, reads the sequence $A_{m, n}$ as a triangle by rows: 1, 1, 1, 2, 4, 2, 6, 56, 56, 6, 24, 1712, 9408, 1712, 24, 120, 92800, 4948992, 4948992, 92800, 120, 720, 7918592, 6085088256, 63352393728, 6085088256, 7918592, 720, $\ldots$.

As with the previous sequence, the sequence of factorials, A000142, is a subsequence of this sequence due to Corollary~\ref{C3}.

\subsection{Chocolate-2 numbers}
As we mentioned before, the first row in Table~\ref{table:f} is the factorials. The next row is a new sequence, \seqnum{A261747} in the OEIS. It lists $B_n=A_{2, n}$, the number of ways to break a 2 by $n$ chocolate bar: 1, 4, 56, 1712, 92800, 7918592, 984237056, 168662855680, 
38238313152512, 11106033743298560, 4026844843819663360,  $\ldots$.

\subsection{Chocolate square numbers}

The following new sequence, \seqnum{A257281} in the OEIS, lists the numbers $A_{n, n}$, the number of ways to break a square chocolate bar with side length $n$: 1, 4, 9408, 63352393728, 3947339798331748515840, 
5732998662938820430255187886059028480, \newline
417673987760293241182652126617960927525362518081132298240, $\ldots$.

\section{Divisibility}
\label{sec:div}

Interesting patterns can be found by examining the divisibility of sequences generated by $A_{m,n}$. Table~\ref{table:g} shows the factorizations of $A_{m,n}$, where rows correspond to $m$ and columns correspond to $n$. We see that the numbers are highly divisible; for example, most of the numbers are divisible by 2.

\begin{table}[htbp]
\begin{center}
\begin{tabular}{c | c | c | c | c |}
 $A_{m, n}$ & \bfseries 1 & \bfseries 2 & \bfseries 3 & \bfseries 4\\
\hline
\bfseries 1 & 1 & 1 & 2 & $2\cdot 3$\\ 
\hline
\bfseries 2 & 1 & $2^2$ & $2^3\cdot 7$ & $2^4\cdot 107$\\ 
\hline
\bfseries 3 & 2 & $2^3\cdot 7$ & $2^6\cdot 3^1\cdot 7^2$ & $2^{10}\cdot 3^3\cdot 179$\\ 
\hline
\bfseries 4 & $2\cdot 3$ & $2^4\cdot 107$ & $2^{10}\cdot 3^3\cdot 179$ & $2^{12}\cdot 3\cdot 13\cdot 19\cdot 20873$\\ 
\hline
\end{tabular}
\caption{$A_{m, n}$ factored}\label{table:g}
\end{center}
\end{table}

The following theorem gives a lower bound for the 2-adic value of a chocolate number. The \textit{2-adic value} of $n$ is defined as the highest power of 2 that divides $n$ and is denoted by $\nu_2(n)$. 

\begin{theorem}
\label{Th5}
If $m, n>1$, then $\nu_2(A_{m, n})\ge m+n-2$.
\end{theorem}

\begin{proof}
We proceed by induction. First consider the base case $m = 2$.  Let $B_n=A_{2, n}$.

By Theorem~\ref{Th2} and Corollary~\ref{C3}, we get 
\[\begin{aligned}
B_n &= A_{2, n}={2n-2 \choose n-1}A_{1, n}A_{n, 1}+\sum_{i=1}^{n-1} {2n-2 \choose 2i-1}A_{2, i}A_{2, n-i}\\
&=(2n-2)!+\sum_{i=1}^{n-1} {2n-2 \choose 2i-1}B_{i}B_{ n-i}.
\end{aligned}\]
By brute force we can check that $\nu_2(B_1)=0$ and $\nu_2(B_j)=j$ for $j=2, 3, 4$.

We induct on $n$.  The base cases $n=2, 3, 4$ have already been verified.  Now assume the result for $i=2, 3, \ldots, n-1$ for some $n>4$.  For all $n>4$, note that $\nu_2((2n-2)!)>n$.  Now note that every term in the summation ${2n-2 \choose 2m-1}B_mB_{n-m}$ has at least the number of powers of $2$ as $B_mB_{n-m}$, which by the induction hypothesis is at least $n$ for $m\neq 1, n-1$.   But for the values $m=1, n-1$, the values of the product $B_mB_{n-m}$  are equal to $B_{n-1}$ and by the inductive hypothesis they each have at least $n-1$ powers of $2$, so added together they have at least $n$ powers of $2$.  Thus we have that $B_n$ is a sum of integers each with at least $n$ powers of $2$, so $\nu_2(B_n)\ge n$ and the induction is complete.  

Similarly, the theorem holds for $n=2$.  

Now assume the statement holds for all $A_{m', n'}$ where $1<m'<m$ and $1<n'<n$ for some $m$, $n>2$.  We prove the statement for $A_{m, n}$.

Recall that\[A_{m,n}=\sum_{i=1}^{m-1} {mn-2 \choose in-1}A_{i, n}A_{m-i, n}+\sum_{i=1}^{n-1} {mn-2 \choose im-1}A_{m, i}A_{m, n-i}.\]

Looking at the first summand, note that for all $i\neq 1$, $m-1$, we have that 
\[
\begin{aligned}
\nu_2\left({mn-2 \choose in-1}A_{i, n}A_{m-i, n}\right) & \ge \nu_2(A_{i, n}A_{m-i, n})\\
& \ge (i+n-2)+(m-i+n-2)\ge m+n-2
\end{aligned}
\]
as $n>2$.  For the other two terms $i=1$, $m-1$, we have 
\[\begin{aligned}
\nu_2&\left(\binom{mn-2}{n-1}A_{1, n}A_{m-1, n}+\binom{mn-2}{mn-n-1}A_{m-1, n}A_{1, n}\right) \\
&\hspace{2cm}=\nu_2\left(2{mn-2 \choose n-1}A_{1, n}A_{m-1, n}\right)\\
&\hspace{2cm}\ge \nu_2\left(2A_{m-1, n}\right)\\
&\hspace{2cm}\ge 1+(m-1+n-2)=m+n-2,
\end{aligned}\]
by the inductive hypothesis.  Thus $\sum\limits_{i=1}^{m-1} {mn-2 \choose in-1}A_{i, n}A_{m-i, n}$ is a sum of terms with $2$-adic order at least $m+n-2$, so $\nu_2 \left(\sum\limits_{i=1}^{m-1} {mn-2 \choose in-1}A_{i, n}A_{m-i, n}\right)\ge m+n-2$.  Similarly, $\nu_2 \left(\sum\limits_{i=1}^{n-1} {mn-2 \choose im-1}A_{m, i}A_{m, n-i}\right)\ge m+n-2$, so $\nu_2(A_{m, n})\ge m+n-2$, as desired, completing the inductive step.

We have already proven the result for $m=2$ and $n=2$.  By the inductive step, if the statement is true for $n=2$ and $m=i$ for all $2\le i\le k$ for some $k\ge 2$, then the statement is true for $A_{k+1, 3}$ and by induction for all $A_{k+1, n}$ for $n>2$.  Thus, by induction the statement is true for all $m$, $n>2$, as desired. 
\end{proof}

In particular, it follows that all Chocolate Numbers except the first one are divisible by 2. We can also estimate 2-adic values for Chocolate-2 numbers and Chocolate Square numbers.

\begin{corollary}
\label{C6}
For the sequence $B_n$ of Chocolate-2 numbers the 2-adic value is bounded below as in: $\nu_2(B_n) \ge n$, where $n>1$.
\end{corollary}

\begin{corollary}
\label{C7}
For the sequence $a(n)$ of Chocolate Square numbers the 2-adic value is bounded below as in: $\nu_2(a(n)) \ge 2n-2$.
\end{corollary}

\section{Divisibility properties of Chocolate-2 numbers}
\label{sec:div2}

Now let us consider a particular sequence $\{B_n\}$, where $B_n = A_{2, n}$, or the number of ways to break a $2\times n$ chocolate bar.   First, recall the following recursion:

\begin{theorem}
\label{Th8}
\[B_n=(2n-2)!+\sum_{m=1}^{n-1} \binom{2n-2}{2m-1} B_mB_{n-m}.\]
\end{theorem}

\begin{proof}
This follows from Theorem~\ref{Th2} and was proven in the proof of Theorem~\ref{Th5}.
\end{proof}

Let us consider the prime factorizations of the Chocolate-2 numbers: $1$, $2^2$, $2^3\cdot 7$, $2^4\cdot 107$, $2^7\cdot 5^2\cdot 29$, $2^{10}\cdot 11\cdot 19\cdot 37$, $2^{10}\cdot 11\cdot 59\cdot 1481$, $2^{12}\cdot 5\cdot 11\cdot 31\cdot 24151$, $2^{15}\cdot 11\cdot 571\cdot 185789$, $2^{17}\cdot 5\cdot 11\cdot 1607\cdot 958673$, $2^{18}\cdot 5\cdot 11\cdot 97\cdot 9371\cdot 307259, \ldots$ 

We already saw that $B_n$ is divisible by $2^n$. Looking at the sequence we see that these numbers have many other divisors. For example, we see that many of the elements are divisible by 5 and/or 11.
We can show that if sufficiently many consebreakive terms are divisible by some number, all the following terms will also be divisible by that number.

\begin{theorem}
\label{Th9}
Let $n$ be a positive integer.  If $k$ is a positive integer dividing $B_i$ for all  $\Bigl\lfloor\dfrac{n+1}{2}\Bigr\rfloor\le i\le n-1$ and satisfying $k|(2n-2)!$, then $k|B_j$ for all $j\ge n$.
\end{theorem}

\begin{proof}
We use induction on $n$.  Since $k$ divides each of the factors in the summation as well as $(2n-2)!$, it must divide $B_n$.  Now assuming the result for $n, n+1, \ldots, m$, we prove the result for $m+1$.  Since $m\ge n$, we have $k|(2m-2)!$, and since $k|B_j$ for all $\Bigl\lfloor\dfrac{n+1}{2}\Bigr\rfloor\le k\le n-1$, so by the same argument as above, $k$ must divide $B_{m+1}$, so we are done.
\end{proof}

This allows us to see when the remainder converges to $0$ when $B_n$ is divided by certain numbers.

We already noticed that many of the terms are divisible by $5$ and $11$.  Then, using the previous theorem, we get the following corollaries.

\begin{corollary}
\label{C10}
For all $i\ge 6$, $11|B_i$.
\end{corollary}

\begin{proof}
Through checking that $11$ divides $B_3$, $B_4$, and $B_5$, the result is immediate by Theorem~\ref{Th9}.
\end{proof}

\begin{corollary}
\label{C11}
For all $i\ge 13$, $5|B_i$.
\end{corollary}

\begin{proof}
By computation, $5|B_i$ for all $13\le i\le 24$.  Thus, by Theorem~\ref{Th9}, the result follows.
\end{proof}

\subsection{Generating functions}

We implement the technique of generating functions \cite{Wilf} to find an explicit form for the generating function of $B_n$ using hypergeometric functions. 

Consider a function 
$$f(X) =  \sum_{n \ge 1} \frac{B_n}{(2n-1)!}X^n.$$

\begin{lemma}
\label{L12}
Function $f(X)$ satisfies the following differential equation:
$$f(X)^\prime = \frac{1}{2(1-X)} + \frac{1}{2X}f(X)+ \frac{1}{2X}f(X)^2.$$
\end{lemma}

\begin{proof}
Indeed,
\[\begin{aligned}
f(X)^2 &=  \left(\sum_{i \ge 1}\frac{B_i}{(2i-1)!}X^i\right) \left(\sum_{j \ge 1}\frac{B_j}{(2j-1)!}X^j\right)
\\
&= \sum_{n\ge 2} \sum_{m=1}^{n-1}\frac{B_mB_{n-m}X^n}{(2m-1)!(2n-2m-1)!}.
\end{aligned}
\]

Notice that Theorem~\ref{Th8} is equivalent to 
\[\frac{B_n}{(2n-2)!} =1+\sum_{m=1}^{n-1} \frac{ B_mB_{n-m}}{(2m-1)!(2n-2m-1)!}.\]
After substituting we get
\[
\begin{aligned}
f(X)^2 
& = \sum_{n\ge 2} \frac{B_nX^n}{(2n-2)!} - \sum_{n\ge 2} X^n\\
& = \sum_{n\ge 1} \frac{B_nX^n}{(2n-2)!}  -  \sum_{n\ge 1} X^n\\
& =\sum_{n\ge 1} \frac{(2n-1)B_nX^{n}}{(2n-1)!}  - \frac{X}{1-X}\\
& = 2\sum_{n\ge 1} \frac{nB_nX^{n}}{(2n-1)!} - f(X)  - \frac{X}{1-X}.
\end{aligned}
\]
On the other hand,
\[
f(X)^\prime= \sum_{n\ge 1} \frac{nB_nX^{n-1}}{(2n-1)!}= \frac{1}{X}\sum_{n\ge 1} \frac{nB_nX^{n}}{(2n-1)!}.
\]
Combining these equations together we get 
\[
f(X)^2= 2Xf'(X) - f(X) - \frac{X}{1-X}.
\]
In other words:
\[
f(X)^\prime = \frac{1}{2(1-X)} + \frac{1}{2X}f(X)+ \frac{1}{2X}f(X)^2.
\]
\end{proof}

This is a \textit{Riccati equation}. Standard methods allow to reduce it to a second order linear ordinary differential equation \cite{Ince}. Namely, we introduce a new function $v(X) = \frac{1}{2X}f(X)$. Then

\[
v^\prime = -\frac{f}{2X^2} + \frac{f^\prime}{2X} =- \frac{f}{2X^2} + \frac{1}{4X(1-X)} + \frac{f}{4X^2}+ \frac{f^2}{4X^2}  =  \frac{1}{4X(1-X)} - \frac{v}{2X}+ v^2.
\]

The standard method of solving an equation of type 
\[v'=v^2 + R(x)v +S(x)
\]
is by substituting $v=-u'/u$ \cite{Ince}. Indeed,
\[v'=-(u'/u)'=-(u''/u) +(u'/u)^2=-(u''/u)+v^2.\]
Therefore,
\[u''/u= v^2 -v'=-S -Rv=-S +Ru'/u\]
and 
\[
 u'' -Ru' +Su=0.\]
 
In our case, 
\begin{equation}
\label{eq:1}
u'' +\frac{u'}{2X} +\frac{u}{4X(1-X)}=0.
\end{equation}

\subsection{Hypergeometric functions}
\label{sec:hyper}

The hypergeometric function is defined for $|X| < 1$ by the power series \cite{Bailey}
\[{}_2F_1(a,b;c;X) = \sum_{n=0}^\infty \frac{(a)_n (b)_n}{(c)_n} \frac{X^n}{n!}.\]

Here $(q)_n$ is called \textit{the rising Pochhammer symbol}, and is defined by:
\[(q)_n = \begin{cases} 1 & n = 0 \\ q(q+1) \cdots (q+n-1) & n > 0 \end{cases}.\]
    
Note that the hypergeometric function is undefined (or infinite) if $c$ equals a non-positive integer.  

The useful fact about the hypergeometric function is that it is a solution of Euler's hypergeometric differential equation \cite{Bailey}
\[X(1-X) w'' + (c-(a+b+1)X) w' - ab\,w = 0. \]

Now we are ready to express function $u$ from Eq.~(\ref{eq:1}) in the following lemma.

\begin{lemma}
\label{L13}
\[u = {}_2F_1\left(\frac{-1-\sqrt{5}}{4}, \frac{-1+\sqrt{5}}{4}; \frac{1}{2}; X\right).
\]
\end{lemma}

\begin{proof}
The equation for $u$ can be rewritten as
\[
X(1-X)u'' +(1/2 -X/2)u' +u/4=0.
\]
This means $c=1/2$, $a+b =-1/2$ and $ab = -1/4$.
\end{proof}

From here $f=-2Xu'/u$. Therefore, the following theorem immediately follows.

\begin{theorem}
\label{Th14}
\[f(X) = -2X \frac{\partial}{\partial X} \log\left(_2F_1\left(\frac{-1-\sqrt{5}}{4}, \frac{-1+\sqrt{5}}{4}; \frac{1}{2}; X\right) \right).\]
\end{theorem}

Now we will use the formula for the hypergeometric function to look at divisibility results modulo different prime numbers. But first we rearrange the series:

\[
\begin{aligned}
_2F_1 & \left(\frac{-1-\sqrt{5}}{4}, \frac{-1+\sqrt{5}}{4}; \frac{1}{2}; X\right) = \sum_{n =0}^\infty \frac{(\frac{-1-\sqrt 5}{4})_n (\frac{-1+\sqrt 5}{4})_n} {(\frac{1}{2})_n}\frac{X^n}{n!} 
\\&\hspace{1cm} = \sum_{n=0}^\infty \frac{\prod_{i=1}^n ((4i-5) - \sqrt{5})\prod_{i=1}^n ((4i-5) + \sqrt{5})}{8^n\prod_{i=1}^n(2i-1)} \frac{X^n}{n!}
\\&\hspace{1cm} = \sum_{n=0}^\infty \frac{\prod_{i=1}^n ((4i-5)^2 - 5)}{4^n(2n)!} X^n.
\end{aligned}
\]

Let us denote the numerators by $P_n$:
\[
P_n = \prod_{i=1}^n ((4i-5)^2 - 5).
\]

\begin{lemma}
\label{L15}
The sequence $P_n$ considered modulo prime $p$ converges to all-zero sequence if and only if $p$ equals 2, 5 or $p \equiv \pm 1 \pmod{5}$.
\end{lemma}

\begin{proof}
If an element $P_i=0$ modulo $p$ for some index $i$, then the sequence is all zeros for all the consebreakive indices. That means the sequence $P_n$ considered modulo prime $p$ converges to all-zero sequence if and only if there exists $i$ such that $p| (4i-5)^2-5)$. 

Notice that $P_1 = -4$ and $P_5$ is divisible by 5. That means the sequence converges when $p$ equals 2 or 5.  

Otherwise, let us assume that $p$ is an odd prime not equal to $5$.  Note that if there exists a positive integer $x$ such that $x^2\equiv 5\pmod p$, then given that $p$ is not 2, there exists $i$ such that $i\equiv (x+5)\cdot 4^{-1}\pmod p$.  Then $x= (4i-5)$.  If there is no such $x$, then clearly we cannot have $(4i-5)^2\equiv 5\pmod p$.  Thus there exists such an $i$ if and only if $5$ is a quadratic residue $\pmod p$.  Since $p$ is an odd prime not equal to $5$, by the Law of Quadratic Reciprocity \cite{HW}, we have $\left(\frac{5}{p}\right)\left(\frac{p}{5}\right)=(-1)^{p-1}=1$, so $5$ is a quadratic residue $\pmod p$ if and only if $p$ is a quadratic residue$\pmod 5$, which is equivalent to $p\equiv 1$ or $4\pmod 5$.
\end{proof}

The sequence of primes congruent to 1 or 4 modulo 5 is sequence \seqnum{A045468} in the database: 11, 19, 29, 31, 41, 59, $\ldots$. This sequence can also be defined as  primes $p$ that divide $F_{p-1}$, the Fibonacci number with index $p-1$. The connection to Fibonacci numbers is not very surprising as the first two parameters of our hypergeometric function are the golden ratio and its inverse, both scaled by $-1/2$.

We ran several programs which computed Chocolate Numbers and evaluated them modulo different primes. It seems that the sequences $B_n$ and $P_n$ eventually become periodic for the same primes.

This motivates the following conjecture.

\begin{conjecture}
\label{Conj1}
The sequence $B_n$ considered modulo prime $p$ converges to all-zero sequence if and only if $p$ equals 2, 5 or $p \equiv \pm 1 \pmod{5}$.
\end{conjecture}

Corollary~\ref{C6} proves that $B_n$ eventually becomes even. Corollary~\ref{C10} proves that $B_n$ eventually becomes divisible by 11, and Corollary~\ref{C11} proves that $B_n$ eventually becomes divisible by 5. Thus we have a proof for $p$ equal to 2, 5, and the smallest number in \seqnum{A045468}.

\subsection{Eventual Periodicity of $B_n$}

We saw that for many prime numbers, $B_n$ becomes divisible by them.  The following theorem shows that $B_n$ is never divisible by $3$. 

\begin{theorem}
\label{Th16}
For all $n>1$, $B_n\equiv 1 \pmod 3$ if $n\equiv 2 \pmod 3$ and $B_n\equiv 2 \pmod 3$ if $n\equiv 0$ or $1 \pmod 3$.
\end{theorem}

We will need two lemmas for this proof.
\begin{lemma}
\label{L17}
For all positive integers $n\equiv 2\pmod 6$ with $n>2$, \[\dbinom{n}{1}+\dbinom{n}{7}+\cdots +\dbinom{n}{n-1}\equiv 1\pmod 3.\]
\end{lemma}
\begin{proof}
We use induction on $n$.  For the base case of $n=8$, the result clearly holds.  Suppose it holds for some $n=k\ge 8$ with $k\equiv 2\pmod 6$; we prove the result for $n=k+6$.  Note that by Pascal's Identity, for all positive integers $m\ge 3$ and $i\ge 3$ and $m\ge i+3$, we have

\begin{equation}
\begin{split}
\dbinom{m}{i} & =\dbinom{m-1}{i-1}+\dbinom{m-1}{i} \\
              & =\dbinom{m-2}{i-2}+2\dbinom{m-2}{i-1}=\dbinom{m-2}{i} \\
              & =\dbinom{m-3}{m-3}+3\dbinom{m-3}{i-2}+3\dbinom{m-3}{i-1}+\dbinom{m-3}{i} \\
              & \equiv\dbinom{k-3}{i-3}+\dbinom{k-3}{i}\pmod 3.
\end{split}
\end{equation}

Applying this to $m=k+6$ and $i=7$, $13$, $\ldots$, $k-1$, we get 
\begin{multline*}
\dbinom{k+6}{1}+\dbinom{k+6}{7}+\cdots +\dbinom{k+6}{k+5}\\
\equiv 
\dbinom{k+3}{1}+\dbinom{k+3}{4}+\dbinom{k+3}{7}+\cdots +\dbinom{k+3}{k-1}+\dbinom{k+3}{k+2}\pmod 3.
\end{multline*}

By the roots of unity filter, we know that the RHS is equal to $\dfrac{2^{k+3}-2}{3}$.  Since $k\equiv 2\pmod 6$, we have $2^{k+3}\equiv 5\pmod 9$, so $\text{LHS}\equiv \text{RHS}\equiv \dfrac{2^{k+3}-2}{3}\equiv 1\pmod 3$, as desired.  Thus the inductive step is complete and the lemma is proven.
\end{proof}

\begin{lemma}
\label{L18}
For all positive integers $n\equiv 4\pmod 6$ with $n>4$, \[\dbinom{n}{5}+\dbinom{n}{11}+\cdots +\dbinom{n}{n-5}\equiv 0\pmod 3\]
\end{lemma}
\begin{proof}
We use induction on $n$.  For the base case of $n=10$, the result clearly holds.  Suppose it holds for some $n=k\ge 10$ with $k\equiv 4\pmod 6$; we prove the result for $n=k+6$.  As shown in Lemma \ref{L17}, we have that for all positive integers $m\ge 3$ and $i\ge 3$ and $m\ge i+3$, $\dbinom{m}{i}\equiv\dbinom{m-3}{i-3}+\dbinom{m-3}{i}\pmod 3$.  Applying this to $m=k+6$ and $i=5$, $11$, $\ldots$, $k+1$, we get 
\begin{multline*}
\dbinom{k+6}{5}+\dbinom{k+6}{11}+\cdots +\dbinom{k+6}{k+1}\\
\equiv \dbinom{k+3}{2}+\dbinom{k+3}{5}+\dbinom{k+3}{8}+\cdots +\dbinom{k+3}{k-2}+\dbinom{k+3}{k+1}\pmod 3.
\end{multline*}

By the roots of unity filter, we know that the RHS is equal to $\dfrac{2^{k+3}-2}{3}$.  Since $k\equiv 4\pmod 6$, we have $2^{k+3}\equiv 2\pmod 9$, so $\text{LHS}\equiv \text{RHS}\equiv \dfrac{2^{k+3}-2}{3}\equiv 0\pmod 3$, as desired.  Thus the inductive step is complete and the lemma is proven.
\end{proof}

Now we are ready to prove Theorem~\ref{Th16}.

\begin{proof}[Proof of Theorem~\ref{Th16}]
We use induction on $n$.  The result can be verified by calculation for all $2\le n\le 12$.  We now prove the result for $n$ with three cases, in each of which we assume the result for all numbers less than $n$.

{\bfseries Case 1.} $n\equiv 1 \pmod 3$.

Note that $3|(2n-2)!$, so we can ignore the factorial term.  We now wish to compute $\sum_{m=1}^{n-1} {2n-2 \choose 2m-1}B_mB_{n-m} \pmod 3$.  By the inductive hypothesis, we have that $B_1, B_2, \ldots, B_{n-1} \equiv 1, 1, 2, 2, 1, 2, 2, \ldots, 1, 2, 2, 1, 2 \pmod 3$ respectively.  Pairing these up appropriately, we get that $B_1B_{n-1}\equiv 2 \pmod 3$, and for $2\le i\le n-2$, $B_iB_{n-i}\equiv 1 \pmod 3$.  Thus, we have 
\[
\begin{aligned}
B_n & \equiv \dbinom{2n-2}{1}+\dbinom{2n-2}{2n-1}+\sum_{m=1}^{n-1} {2n-2 \choose 2m-1}\\
&=4n-4+\frac{1}{2}\cdot 2^{2n-2}\equiv 2^{2n-1}\equiv 2 \pmod 3,
\end{aligned}\]
where we used the fact that the sum of the odd binomial coefficients is equal to the sum of the even binomial coefficients.  

{\bfseries Case 2.} $n\equiv 2 \pmod 3$.

Note that $3|(2n-2)!$, so we can ignore the factorial term.  We now wish to compute $\sum_{m=1}^{n-1} {2n-2 \choose 2m-1}B_mB_{n-m} \pmod 3$.  By the inductive hypothesis, we have that $B_1, B_2, \ldots, B_{n-1} \equiv 1, 1, 2, 2, 1, 2, 2, \ldots, 1, 2, 2, 1, 2, 2 \pmod 3$ respectively.  Pairing these up appropriately, we get that for $i\equiv 1 \pmod 3$ and $4\le i\le n-4$, $B_iB_{n-i}\equiv 1\pmod 3$ and $B_iB_{n-i}\equiv 2\pmod 3$ otherwise.  Thus we have modulo 3
\[
\begin{split}
B_n & \equiv 2\sum_{m=1}^{n-1}\dbinom{2n-2}{2m-1}+\dbinom{2n-2}{1}+\dbinom{2n-2}{2n-3}-\sum_{\substack{1\le i\le 2n-3 \\ i \equiv 1 \pmod{6}}}\dbinom{2n-2}{i} \\
    & \equiv 2^{2n-2}+4n-4-\left(\dbinom{2n-2}{5}+\dbinom{2n-2}{11}+\cdots +\dbinom{2n-2}{2n-7}\right) \\
    & \equiv 2-\left(\dbinom{2n-2}{1}+\dbinom{2n-2}{7}+\cdots +\dbinom{2n-2}{2n-3}\right) \\
    & \equiv 1,
\end{split}
\]
with the last congruence following from Lemma~\ref{L17}.

{\bfseries Case 3.} $n\equiv 0 \pmod 3$.

Note that $3|(2n-2)!$, so we can ignore the factorial term.  We now wish to compute $\sum_{m=1}^{n-1} {2n-2 \choose 2m-1}B_mB_{n-m} \pmod 3$.  By the inductive hypothesis, we have that $B_1, B_2, \ldots, B_{n-1} \equiv 1, 1, 2, 2, 1, 2, 2, \ldots, 1, 2, 2, 1 \pmod 3$ respectively.  Pairing these up appropriately, we get that for $i=1, n-1$ and $i\equiv 0\pmod 3$, $B_iB_{n-i}\equiv 1\pmod 3$ and $B_iB_{n-i}\equiv 2\pmod 3$ otherwise.  Thus we have modulo 3
\[\begin{aligned}
B_n & \equiv 2\sum_{m=1}^{n-1}\dbinom{2n-2}{2m-1}-\dbinom{2n-2}{1}-\dbinom{2n-2}{2n-3}-\sum_{\substack{5\le i\le 2n-4 \\ i \equiv 5\mod 6}}\dbinom{2n-2}{i} \\
    & \equiv 2^{2n-2}+4n-4-\left(\dbinom{2n-2}{5}+\dbinom{2n-2}{11}+\cdots +\dbinom{2n-2}{2n-7}\right) \\
    & \equiv 2-\left(\dbinom{2n-2}{5}+\dbinom{2n-2}{11}+\cdots +\dbinom{2n-2}{2n-7}\right) \\
    & \equiv 0,
\end{aligned}\]
with the last congruence following from Lemma~\ref{L18}.

Thus for every remainder modulo 3, the theorem is true.
\end{proof}

We know that the sequence $B_n$ becomes divisible by certain primes for sufficiently large indices.  We know that $P_n$ becomes divisible by the same primes. This motivates us to consider the divisibility properties of both sequences, which in turn leads us to questioning the behavior of $P_n$ modulo any prime.

\begin{lemma}
\label{L19}
The sequence $P_n$ considered modulo prime $p$ eventually becomes periodic.
\end{lemma}

\begin{proof}
Recall that \[
P_n = \prod_{i=1}^n ((4i-5)^2 - 5).
\]By Lemma~\ref{L15}, the desired statement is true for $p=2$, $5$, and $p\equiv\pm 1\pmod 5$.  For all other $p$, we claim that $P_n\equiv P_{n+p(p-1)}$ for all $n$.  Indeed, since $(4i-5)^2-5\equiv (4(i+p)-5)^2-5\pmod p$ for all integers $i$, we know that the remainder when $\prod_{i=k}^{k+p-1} ((4i-5)^2 - 5)$ is divided by $p$ is fixed for all positive integers $k$.  Furthermore, from the proof of Lemma~\ref{L15}, this is not $0$ for these primes.  Letting this remainder be $r$, by Fermat's Little Theorem, we have $r^{p-1}\equiv 1\pmod p$, so $P_{n+p(p-1)}\equiv r^{p-1}P_n\equiv P_n\pmod p$, as desired.
\end{proof}

This argument can be extended to show that $P_n$ eventually becomes periodic when considered modulo any positive integer $m$.

\begin{lemma}
\label{L20}
The sequence $P_n$ considered modulo any positive integer $m$ eventually becomes periodic.
\end{lemma}

\begin{proof}
We first prove the statement for all prime powers $p^t$.  As in the proof of Lemma~\ref{L19}, if $\prod_{i=1}^{p} ((4i-5)^2 - 5)$ is divisible by $p$, then $\prod_{i=k}^{k+p-1} ((4i-5)^2 - 5)$ is divisible by $p$ for all $k$, so $P_{tp}$ will be divisible by $p^t$, and thus every term of the sequence afterwards will also be divisible by $p^t$.  If the remainder when $\prod_{i=k}^{k+p-1} ((4i-5)^2 - 5)$ is divided by $p$ is not 0, then let it be $r$.  By Euler's totient theorem, we have $r^{\phi{p^t}}=r^{(p-1)p^{t-1}}\equiv 1\pmod p^t$, giving $P_{n+(p-1)p^{t-1}}\equiv r^{(p-1)p^{t-1}}P_n\equiv P_n \pmod p^t$ for all positive integers $n$, completing the statement for prime powers.  Thus, splitting each positive integer $m$ into a product of prime powers, the sequence $P_n$ becomes divisible by some of them and is periodic modulo the rest of them, beginning from $P_1$.  Thus, by the Chinese Remainder Theorem, multiplying all the periods together gives that $P_n$ is periodic $\pmod m$ for sufficiently large $n$, as desired.
\end{proof}

The eventual periodicity of $P_n$ makes us expect the eventual periodicity of $B_n$. In addition, we have further computational evidence supporting eventual periodicity modulo any number. Thus we have the following conjecture:

\begin{conjecture}
\label{Conj2}
For any positive integer $m$, for sufficiently large indices, the sequence $B_n$ becomes periodic modulo $m$.
\end{conjecture}

We also ran a lot of computational experiments that allowed us to make the following conjecture.  This is supported by Lemma~\ref{L19}, assuming a connection between the behavior of $B_n$ and $P_n$.

\begin{conjecture}
\label{Conj3}
If $p$ is not $2$, $5$, or equivalent to $\pm 1\pmod 5$, then the sequence $B_n$ modulo $p$ repeats with a period divisible by $p(p-1)$.
\end{conjecture}

\section{Acknowledgements}
We would like to thank Prof.~James Propp for the inspiration behind this paper and several useful discussions. Also, we are grateful to MIT PRIMES for providing the opportunity to conduct this research. Finally, we would like to thank Prof.~Pavel Etingof for a helpful discussion.

\bigskip
\hrule
\bigskip

\noindent 2010 {\it Mathematics Subject Classification}: Primary 11B99.

\noindent \emph{Keywords: } generating functions, hypergeometric functions, periodicity, p-adic order, recursion, Riccati equations.

\bigskip
\hrule
\bigskip

\noindent
(Mentions \seqnum{A000142}, \seqnum{A045468}. New sequences \seqnum{A257281}, \seqnum{A261746}, \seqnum{A261747}, \seqnum{A261964}.)

\end{document}